\documentclass[bj,authoryear]{imsart}

\startlocaldefs
\numberwithin{equation}{section}
\theoremstyle{plain}

\newtheorem{theorem}{Theorem}[section]

\newtheorem{lemma}[theorem]{Lemma}
\newtheorem{proposition}[theorem]{Proposition}

\theoremstyle{definition}

\theoremstyle{definition}
\newtheorem{remark}[theorem]{Remark}

\newcommand{\ind}{{\bf 1}}

\def\inddd#1{{\ind}_{\left\{#1\right\}}} 
\newcommand{\proba}{\mathbb P}
\newcommand{\esp}{{\mathbb E}}

\newcommand{\inv}{^{-1}}
\newcommand{\cov}{{\rm{Cov}}}

\newcommand{\eqnh}{\begin{eqnarray*}}
	\newcommand{\eqne}{\end{eqnarray*}}
\newcommand{\eqnhn}{\begin{eqnarray}}
	\newcommand{\eqnen}{\end{eqnarray}}
\newcommand{\equh}{\begin{equation}}
	\newcommand{\eque}{\end{equation}}

\def\summ#1#2#3{\sum_{#1 = #2}^{#3}}

\def\sif#1#2{\sum_{#1=#2}^\infty}

\newcommand{\eqd}{\stackrel{d}{=}}

\def\topp#1{^{(#1)}}

\def\nn#1{{\left\|#1\right\|}}

\def\abs#1{\left|#1\right|}
\def\sabs#1{|#1|}

\def\ccbb#1{\left\{#1\right\}} 

\def\pp#1{\left(#1\right)}
\def\spp#1{(#1)}

\def\mmid{\;\middle\vert\;}

\def\floor#1{\left\lfloor #1 \right\rfloor}

\def\qmand{\quad\mbox{ and }\quad}

\def\mfa{\mbox{ for all }}

\def\wt#1{\widetilde{#1}}
\def\wb#1{\overline{#1}}
\def\what#1{\widehat{#1}}

\def\limn{\lim_{n\to\infty}}

\def\limsupn{\limsup_{n\to\infty}}

\def\weakto{\Rightarrow}

\def\R{{\mathbb R}}

\def\N{{\mathbb N}}

\newcommand{\calF}{{\mathcal F}}
\newcommand{\calG}{{\mathcal G}}

\newcommand{\calP}{{\mathcal P}}
\newcommand{\calL}{{\mathcal L}}

\renewcommand{\d}{{\rm d}}
\newcommand{\aswto}{\stackrel{a.s.w.}\to}

\date{}

\newcommand{\B}{{\mathbb B}}
\def\R{\mathbb{R}}

\def\N{\mathbb{N}}

\def\1{\mathds{1}}

\endlocaldefs

\begin{document}\sloppy

\begin{frontmatter}
\title{On central limit theorems for the Ewens--Pitman model}
\runtitle{On CLT for the Ewens--Pitman model}

\begin{aug}
\author[A]{\inits{Y.}\fnms{Yizao}~\snm{Wang}
\ead[label=e1]{yizao.wang@uc.edu}}
\address[A]{Department of Mathematical Sciences, University of Cincinnati, 2815 Commons Way, 
Cincinnati, OH, 45221-0025, USA\printead[presep={,\ }]{e1}}
\end{aug}

\begin{abstract}
We establish a quenched functional central limit theorem for the total number of components of random partitions induced by a Chinese restaurant process with parameters $(\alpha,\theta), \alpha\in(0,1), \theta>-\alpha$. With $P_j$ denoting the asymptotic frequency of the $j$-th table, it is well-known that the component count has the same law as the occupancy count of an infinite urn scheme with sampling frequencies being $(P_j)_{j\in\N}$. Our analysis follows this approach and is based on earlier results of \citet{karlin67central} and \citet{durieu16infinite}. In words, our result reveals that the fluctuations of the component count consist of two parts, one due to the sampling effect given the asymptotic frequencies $(P_j)_{j\in\N}$, the other due to the fluctuations of the random asymptotic frequencies, and in the limit the fluctuations of the two parts are conditionally independent given the $\alpha$-diversity. Our result strengthens a recent central limit theorem obtained by \citet{bercu24martingale} via a different method.
\end{abstract}

\begin{keyword}
\kwd{Chinese restaurant process}
\kwd{functional central limit theorem}
\kwd{infinite urn model}
\kwd{random permutation}
\end{keyword}

\end{frontmatter}

	\section{Introduction and main result}
Consider the random partitions induced by a Chinese restaurant process with $(\alpha,\theta)$-seating, with $\alpha\in[0,1), \theta>-\alpha$. This family of exchangeable random partitions, referred to as $(\alpha,\theta)$-partitions in the sequel, is arguably one of the most fundamental models in combinatorial stochastic processes \citep{pitman06combinatorial}. 
In recent literature the $(\alpha,\theta)$-partitions have also been referred to as the Ewens--Pitman model (see \citet{favaro25asymptotic} and some references therein). 

When $\alpha=0$, the law of the exchangeable partitions follows the well-known Ewens sampling formula \citep{crane16ubiquitous} with parameter $\theta>0$, and in the special case $(\alpha,\theta)=(0,1)$ the induced random permutations are uniform. Earlier developments from the combinatorial and probabilistic aspects of the random partitions can be found in \citet{pitman06combinatorial,arratia03logarithmic}. The induced random permutations have been studied more recently \citep{bahier22smooth,benarous15fluctuations,garza24limit,wieand00eigenvalue,francois25characteristic} in the literature of random matrix theory. On the application side, great success of these exchangeable random partitions has been found in Bayesian nonparametrics. 
The asymptotic frequencies $(P_j)_{j\in\N}$ of the $(\alpha,\theta)$-partitions have the 
Griffiths--Engen--McCloskey (GEM) distribution with parameters $(\alpha,\theta)$ and arise as the weights of the Pitman--Yor process; the special case $\alpha=0$, known as the Dirichlet process, was first investigated by \citet{ferguson73bayesian}. In words, the random partitions correspond to random clusters in various hierarchical models built upon the Pitman--Yor process. We refer to \citet{broderick12beta,contardi25laws} and references therein for results on Bayesian nonparametrics related to the $(\alpha,\theta)$-partitions. It is worth pointing out that  the majority of the developments have focused on $\alpha=0$ in the literature. For some recent developments on $\alpha\in(0,1)$, see \citet{favaro25asymptotic,contardi25laws} and references therein.

In this paper, we focus on the case $\alpha\in(0,1),\theta >-\alpha$. 
	Let $K_n$ denote the total number of components of the partition of $\{1,\dots,n\}$. Then, it is well-known that
\equh\label{eq:LLN}
	\limn \frac{K_n}{n^\alpha} = S_{\alpha}\quad \mbox{ almost surely,}
\eque
	where $S_\alpha$ is known as the $\alpha$-diversity. Let $(P_j)_{j\in\N}$ denote the asymptotic frequencies of the tables. 
	Set $\calP :=\sigma((P_j)_{j\in\N})$. It is known that $S_\alpha$ is $\calP$-measurable. In particular, with $(P_j^\downarrow)_{j\in\N}$ denoting the asymptotic frequencies in decreasing order, 
\equh\label{eq:diversity}
S_\alpha \equiv S_{\alpha,\theta} := \lim_{j\to\infty}j(P_j^\downarrow)^\alpha \Gamma(1-\alpha), \mbox{ almost surely.}
\eque
The normalization depends only on $\alpha$, and hence in the notation usually the parameter $\theta$ is omitted (while its law depends obviously on $\theta$ too).
A standard reference for us is \citet[Chapter 3]{pitman06combinatorial}, and detailed definitions are provided in Section \ref{sec:CRP}.
	
	There are two central limit theorems in the literature regarding $K_n$. 
First, 
we have	\equh\label{eq:Karlin}
\frac{K_n - \esp(K_n\mid\calP)}{n^{\alpha/2}}\weakto (2^\alpha-1)^{1/2} S_\alpha^{1/2}Z
\eque
as $n\to\infty$,
	where $\calP = \sigma((P_j)_{j\in\N})$
 and $Z$ is a standard Gaussian random variable independent of $S_\alpha$. Throughout, we let $\weakto$ denote convergence in distribution.
This result can be read already from a result of \citet{karlin67central} and Kingman's representation theorem \citep{kingman78representation}, but  seems to have only been explicitly stated recently in \citet[Theorem 3.1]{garza24limit} (with a few extensions, taking $a_j=1$ therein). In fact, the convergence in \eqref{eq:Karlin} was established in the quenched sense, and a quenched functional central limit theorem has been known \citep{durieu16infinite}; all these will be recalled  in \eqref{eq:DW16} below. It is worth pointing out that the centering by $\esp(K_n\mid\calP)$ is natural in view of Karlin's result (which concerns equivalently an infinite urn scheme to be explained below), although it is not the {\em right} centering in view of \eqref{eq:LLN}. 
	
Second, recently \citet{bercu24martingale} established the following central limit theorem:
	\equh\label{eq:BF}
	n^{\alpha/2}\pp{\frac{K_n}{n^\alpha}-S_\alpha} \equiv\frac{K_n - n^\alpha S_\alpha}{n^{\alpha/2}}\weakto S_\alpha^{1/2}Z
	\eque
	as $n\to\infty$, 
	where on the right-hand side, $S_\alpha$ is independent from the standard Gaussian random variable $Z$. The proof is completely different from the one for \eqref{eq:Karlin}, and in particular it relies on a martingale central limit theorem for the `tail sum' of infinite series due to \citet{heyde77central}. Note that \eqref{eq:BF} uses the {\em natural} centering in view of \eqref{eq:LLN}.

In view of the two previously mentioned results, it is natural to write
	\equh\label{eq:decompose}
	\frac{K_n - n^\alpha S_\alpha}{n^{\alpha/2}} = \frac{K_n - \esp(K_n\mid\calP)}{n^{\alpha/2}} + \frac{\esp(K_n\mid\calP) - n^\alpha S_\alpha}{n^{\alpha/2}},
	\eque
and then to expect a central limit theorem also for the second term on the right-hand side above. The main contribution of this paper is  a functional central limit theorem for the joint convergence of the two terms on the right-hand side above. Our functional central limit theorem concerns the following processes
	\begin{align}
	W_n(t)&:=\frac{K_{\floor{nt}}-\esp(K_{\floor{nt}}\mid\calP)}{n^{\alpha/2}},\nonumber\\
	Y_n(t)&:=\frac{\esp(K_{\floor{nt}}\mid\calP)-\floor{nt}^\alpha S_\alpha}{n^{\alpha/2}}, t\in[0,1],n\in\N,\label{eq:Y}
	\end{align}
and the limit processes involve two centered Gaussian processes denoted by $Z_\alpha\topp1,Z_\alpha\topp2$, of which the covariance functions are as follows:
\begin{align*} 
\cov\pp{Z_\alpha\topp1(s),Z_\alpha\topp1(t)} &= (s+t)^\alpha-\max(s^\alpha,t^\alpha),
\\
\cov\pp{Z_\alpha\topp2(s),Z_\alpha\topp2(t)} &= s^\alpha+t^\alpha-(s+t)^\alpha, \quad s,t>0.
\end{align*}
Note that if $Z_\alpha\topp1$ and $Z_\alpha\topp2$ are independent, then 
\equh\label{eq:decomposition}
\pp{Z_\alpha\topp1(t)+Z_\alpha\topp2(t)}_{t\in[0,1]}\eqd \pp{\B_{t^\alpha}}_{t\in[0,1]},
\eque
where the right-hand side is a time-changed Brownian motion (the standard Brownian motion $\B$ indexed by $t^\alpha$).

The main result of this paper is the following. 	
\begin{theorem}\label{thm:1}
		With the notations above,
		\equh\label{eq:joint CLT}
		\pp{\pp{W_n(t)}_{t\in[0,1]},\pp{Y_n(t)}_{t\in[0,1]}}\weakto \pp{S_\alpha^{1/2}\pp{Z\topp1_\alpha(t)}_{t\in[0,1]},S_\alpha^{1/2}\pp{Z_\alpha\topp2(t)}_{t\in[0,1]}},
		\eque
in $D[0,1]^2$ as $n\to\infty$, where on the right-hand side $Z\topp1_\alpha$ and $Z_\alpha\topp2$ are two independent Gaussian processes introduced above and independent from $S_\alpha$.
	\end{theorem}

Moreover, we shall establish a quenched version of \eqref{eq:joint CLT} in Theorem \ref{thm:2}, from which Theorem \ref{thm:1} follows as an immediate consequence. The $\sigma$-algebras ($\calP_n$ below) involved in Theorem \ref{thm:2} would take a little preparation to introduce, and we skip the details in the introduction here.

A quenched functional central limit theorem for the first component $W_n$ has already been established in \citet[Theorem 2.3]{durieu16infinite} (in combination with Kingman's representation theorem). Namely, it was shown that
\equh\label{eq:DW16}
\pp{W_n(t)}_{t\in[0,1]}\equiv \pp{\frac{K_{\floor{nt}}-\esp(K_{\floor{nt}}\mid\calP)}{n^{\alpha/2}}}_{t\in[0,1]} \aswto \pp{S_\alpha^{1/2}Z_\alpha\topp1(t)}_{t\in[0,1]}
\eque
in $D[0,1]$ with respect to $\calP$ as $n\to\infty$, where $S_\alpha$ is $\calP$-measurable and $Z_\alpha\topp1$ is independent from $\calP$.  Here, the notation $\aswto$ stands for {\em almost sure weak convergence} \citep{grubel16functional}; for the sake of simplicity we also refer to $\aswto$ as a {\em quenched} convergence. We say a sequence of random elements $(X_n)_{n\in\N}$ converges almost surely weakly to $X$ in a certain Polish space $M$ with respect to a $\sigma$-algebra $\calG$, if for all continuous and bounded functions $f:M\to\R$ we have
\[
\limn \esp(f(X_n)\mid\calG) = \esp(f(X)\mid\calG), \mbox{ almost surely.}
\]
When writing $X_n\aswto X$ with respect to $\calG$, implicitly we assume all $(X_n)_{n\in\N}$ and $X$ to be defined on a probability space of which $\calG$ is a $\sigma$-algebra; this assumption is needed for the conditional expectations above to be well-defined. This is different from the interpretation of weak convergence, when only the laws of the random variables are concerned and hence the underlying probability space is irrelevant.

In particular, in all our statements regarding quenched convergence, all the pre-limit statistics $W_n(t),Y_n(t)$ and $S_\alpha$ (in the pre-limit statistic and in the limit process) are defined on a common probability space.

We comment briefly on the proof. It is well-known that by Kingman's representation theorem, the study of exchangeable random partitions with asymptotic frequencies $(P_j)_{j\in\N}$ can be translated into the study of random partitions induced by an infinite urn scheme (paintbox partitions) with sampling frequencies $(P_j)_{j\in\N}$, which in decreasing order decay at a polynomial rate with tail index $-1/\alpha$. 
In the language of Bayesian nonparametrics, the latter construction gives rise to random partitions generated by i.i.d.~sampling from the random frequencies, whose decreasing arrangement follows the Poisson--Dirichlet distribution with parameters $(\alpha,\theta)$ (and the Chinese restaurant process is involved in practice); a closely related object is the Pitman--Yor process, taking the form of a random probability measure $\sif j1 P_j\delta_{V_j}$ where $V_j$ are i.i.d.~from a base distribution, independent from $(P_j)_{j\in\N}$, and when $\alpha=0$ this is  known as the Dirichlet process.

The infinite urn scheme is another fundamental model in probability theory with early developments dating back to the 1960s \citep{bahadur60number,karlin67central,darling67some}; see \citet{gnedin07notes} and references therein for early developments. 
Some recent works on the component $W_n(t)$ that essentially exploit the (conditional) i.i.d.~structure include \citet{durieu16infinite,chebunin16functional,iksanov22functional}.
In words, most of the previous analysis exploiting the connection to the infinite urn scheme only needs the assumption that the asymptotic frequencies decay polynomially {\em almost surely}; that is,  $P_j^\downarrow\sim d_0j^{-1/\alpha}$ as $j\to\infty$. The Gaussian fluctuation in the limit on the first component $W_n$ is due to the sampling procedure {\em given the sampling frequencies}, and the deviation of $P_j^\downarrow$ from $d_0j^{-1/\alpha}$ is not captured in the limit fluctuation (in fact, the deviation is not considered when setting up the question). This part is essentially due to Karlin, who first extensively investigated limit theorems for an infinite urn scheme with polynomially decaying sampling frequencies.  The model with such sampling frequencies is referred to as the Karlin model in the literature.

Here, our analysis essentially examines the fluctuations of $P_j^\downarrow$ around $d_0j^{-1/\alpha}$. This fluctuation leads to the Gaussian fluctuation in the second component $Y_n(t)$ of the main result. This result is quite different from all the aforementioned ones on the Karlin model. At the heuristic level it is clear that the two fluctuations should be conditionally independent. Our Theorem \ref{thm:2} explains this in more detail by providing a specific choice of the $\sigma$-algebra $\calP_n$ to condition on. On the other hand, it is remarkable  that the approach by \citet{bercu24martingale} does not exploit the urn scheme connection at all, although it is not clear whether it could deal with $\esp (K_n\mid\calP)$ in order to have access to the limit of the decomposition or the quenched convergence.

\begin{remark}
The decomposition \eqref{eq:decomposition} appeared already in \citet{durieu16infinite}, and it is worth pointing out that a corresponding functional central limit theorem for the decomposition was established. The model investigated therein was the Karlin model (not necessarily the Chinese restaurant process) randomized by Rademacher random variables. Stochastic-integral representations of the processes $Z_\alpha\topp1,Z_\alpha\topp2$ can be found in \citet{fu20stable}. 
\end{remark}
\begin{remark}
Our methodology can be further extended to study the total number of components in $\Pi_n$ with exactly $j$ elements, denoted by $C_{n,j}$ below. In the language of random permutations $C_{n,j}$ is referred to as the $j$-cycle count in the literature. It is well-known that
\[
\limn \frac{C_{n,j}}{n^\alpha} = S_\alpha p\topp\alpha_j:=S_\alpha\frac{\alpha\Gamma(j-\alpha)}{\Gamma(1-\alpha)\Gamma(j+1)}, \mbox{ almost surely, for $j\in\N$,}
\]
and $(p_j\topp\alpha)_{j\in\N}$ is the probability mass function of the $\alpha$-Sibuya distribution.
Central limit theorems regarding $C_{n,j}$ have also been studied in the literature. Again, the choice of the centering is delicate, and we could write
\[
\frac{C_{n,j}-S_\alpha p\topp\alpha_jn^\alpha}{n^{\alpha/2}} = \frac{C_{n,j}-\esp(C_{n,j}\mid\calP)}{n^{\alpha/2}}+\frac{\esp(C_{n,j}\mid\calP)-S_\alpha p\topp\alpha_jn^\alpha}{n^{\alpha/2}}.
\]
The convergence of the left-hand side above  was established in \citet{bercu24martingale}, and the quenched joint convergence of the first statistic on the right-hand side was established
 in \citet{karlin67central,chebunin16functional}  (the convergence is joint for random variables indexed by $j\in\N$, and the joint convergence was in fact established for a functional version with $n$ replaced by $\floor{nt}$). This extension was recently addressed by \citet{garza26second}, who investigated this question both for $j$ fixed and for $j = j_n\to\infty$ at an appropriate rate; in fact, depending on the rate $j_n\to\infty$ a phase transition was revealed therein (see also \citet{banderier24phase}).  In fact, we expect a more general quenched functional central limit theorem for the following sequence of bivariate processes
 \[
 \pp{\frac1{n^{\alpha/2}}\sif j1 a_j\pp{C_{\floor{nt},j}-\esp(C_{\floor{nt},j}\mid\calP)}, \frac1{n^{\alpha/2}}\sif j1 b_j\pp{\esp(C_{\floor{nt},j}\mid\calP)-S_\alpha p\topp\alpha_j(\floor{nt})^\alpha}}_{t\in[0,1]}
 \]
 for suitable constants $(a_j)_{j\in\N}, (b_j)_{j\in\N}$. The quenched convergence of the first component has been shown in \citet{garza25functional}, with motivations from random permutation matrices; therein statistics of interest often take the form of $\sif j1 a_jC_{n,j}$ \citep{garza24limit}. The challenge is to find a not-too-restrictive condition on $(a_j)_{j\in\N}$ and $(b_j)_{j\in\N}$ when establishing the tightness. This is left for future research. 
\end{remark}

\begin{remark}
The central limit theorem for $K_n$ for the case $\alpha=0$ is well known. See \citet{arratia03logarithmic}. Note that \citet{hansen90functional} established a functional central limit theorem, although the way the parameter $t$ was introduced there is not the same. 
\end{remark}

\begin{remark}\label{rem:literature}
We mention a few more related results in the literature. For a general class of infinite urn models with random sampling frequencies known as the Bernoulli sieve (where $P_j = (1-\xi_1)\cdots(1-\xi_{j-1})\xi_j$ with i.i.d.~random variables $(\xi_k)_{k\in\N}$ taking values in $(0,1)$), the same decomposition was investigated by \citet{gnedin10limit}. Therein, the random frequencies were viewed as a random environment, the fluctuations from which were shown to be dominant in the limit under a suitable assumption on $\xi_1$. Note that the Bernoulli sieve includes $(0,\theta)$-partitions, but not the $(\alpha,\theta)$-ones with $\alpha\in(0,1)$. A more sophisticated urn model with random sampling frequencies was studied in \citet{gnedin12regenerative}, where the random frequencies are related to a subordinator with a L\'evy measure slowly varying at zero and again the random environment has the dominant contribution to the limit fluctuations. At a higher level, we mention \citet{iksanov17asymptotics}, who investigated a random process with immigration at the epochs of a renewal process. The fluctuations of the process have two sources, one from the immigration and the other from the underlying renewal process determining when the immigration occurs. For a certain range of the parameters, it was shown that the two parts are of comparable orders and both contribute in the limit. 
\end{remark}

The paper is organized as follows. In Section \ref{sec:CRP} we provide preliminary results on the $(\alpha,\theta)$-partitions. In Section \ref{sec:proof} we state and prove the stronger quenched functional central limit theorem in Theorem \ref{thm:2}.


\section{Preliminary results on the Chinese restaurant process}\label{sec:CRP}

We first recall the Chinese restaurant process. The standard reference is \citet{pitman06combinatorial}.  The process  has two parameters $(\alpha,\theta)$ with $\alpha\in[0,1)$ and $\theta>-\alpha$ and consists of a family of exchangeable random partitions $(\Pi_n)_{n\in\N}$, each of $[n]=\{1,\dots,n\}$, constructed consecutively. 
 The procedure goes as follows. Set $\Pi_1 = \{\{1\}\}$. Suppose $\Pi_n = \{\Pi_{n,1},\dots,\Pi_{n,k}\}$ ($(\Pi_{n,j})_{j=1,\dots,k}$ are disjoint non-empty subsets of $[n]$ and $\bigcup_{j=1}^k \Pi_{n,j} = [n]$; in this case $\Pi_n$ is said to have $k$ components). Then, the partition $\Pi_{n+1}$ is obtained by 
\begin{enumerate}
\item[(i)] adding element $n+1$ to an existing block $j$ (i.e., setting $\Pi_{n+1,j} := \Pi_{n,j}\cup \{n+1\}$) with probability $(|\Pi_{n,j}|-\alpha)/(n+\theta)$; 
\item[(ii)] creating a new block with a single element $n+1$ (i.e., setting $\Pi_{n+1,k+1} := \{n+1\})$ with probability $(k\alpha+\theta)/(n+\theta)$;
\item[(iii)] all other existing blocks remain unchanged (i.e., setting $\Pi_{n+1,j} = \Pi_{n,j}$, for all $j=1,\dots,k$ that have not been involved). 
\end{enumerate}
The statistic of interest is the total number of components of the partition $\Pi_n$, denoted by $K_n$ throughout. To state our main result, the asymptotic frequencies are involved. It is well-known that
\[
P_j\equiv P_j\topp{\alpha,\theta}:=\limn \frac{|\Pi_{n,j}|}n \quad \mbox{  exists almost surely,}
\]
for every $j\in\N$.
The law of $(P_j)_{j\in\N}$ is known as the Griffiths--Engen--McCloskey (GEM) distribution with parameters $(\alpha,\theta)$. For a recent generalization, see \citet{basrak25generalized}.
The law of the decreasingly ordered sequence $(P_j^\downarrow)_{j\in\N}$ is known as the Poisson--Dirichlet distribution with parameters $(\alpha,\theta)$, denoted by $\mathsf P_{\alpha,\theta}$ below \citep{feng10poisson}.  Recall the definition of $S_\alpha$ in \eqref{eq:diversity}. 

It is well-known that $P_j^\downarrow\sim d_0 j^{-1/\alpha}$ with $d_0 = (S_\alpha/\Gamma(1-\alpha))^{1/\alpha}$. This was first established via an intrinsic relation to the jumps of an $\alpha$-stable subordinator. We need to exploit this fact further to have an estimate on the deviation of $P_j^\downarrow$ from $d_0j^{-1/\alpha}$. The following can be read from \citet{pitman06combinatorial}. 
\begin{lemma}\label{lem:P_j}
Let $(P_j^\downarrow)_{j\in\N}$ follow the Poisson--Dirichlet distribution with parameters $(\alpha,\theta)$, and $S_\alpha$ be as in \eqref{eq:diversity}. Set
\equh\label{eq:Gamma_j}
\Gamma_j:=\frac{S_{\alpha}}{\Gamma(1-\alpha)}\pp{P_j^\downarrow}^{-\alpha}, j\in\N.
\eque
\begin{enumerate}
\item[(i)] When $\theta=0$, the sequence $(\Gamma_j)_{j\in\N}$ has the law of consecutive arrival times of a standard Poisson process.
\item[(ii)] More generally for all $\theta>-\alpha$, the law of $\mathsf P_{\alpha,\theta}$ is absolutely continuous with respect to $\mathsf P_{\alpha,0}$. As a consequence, all the almost sure statements regarding $\mathsf P_{\alpha,0}$ remain valid for $\mathsf P_{\alpha,\theta}$.
\end{enumerate}
In particular, it follows from \eqref{eq:diversity} and \eqref{eq:Gamma_j} that, for all $\alpha\in(0,1)$ and $\theta>-\alpha$, as $j\to\infty$
\[
P_j^\downarrow =  \pp{\frac{S_{\alpha}}{\Gamma(1-\alpha)}}^{1/\alpha} \Gamma_j^{-1/\alpha}    \sim \pp{\frac{S_{\alpha}}{\Gamma(1-\alpha)}}^{1/\alpha} j^{-1/\alpha}, \mbox{ almost surely.}
\]
\end{lemma}
\begin{proof}
We first prove part (i). Assume $\theta=0$. Let $(\wt\Gamma_j)_{j\in\N}$ denote the consecutive sequence of arrival times of a standard Poisson process.  Recall that one can express the law of the ordered asymptotic frequencies $(P_j^\downarrow)_{j\in\N}$ in terms of jumps from an $\alpha$-stable subordinator: 
\[
\pp{P_j^\downarrow}_{j\in\N}\eqd\pp{\frac{\wt\Gamma_j^{-1/\alpha}}{\sif k1\wt\Gamma_k^{-1/\alpha}}}_{j\in\N}.
\]
See \citet{pitman97two,perman92size}. The above then implies that
\[
\pp{\pp{P_j^\downarrow}_{j\in\N},\lim_{j\to\infty}j(P_j^\downarrow)^\alpha \Gamma(1-\alpha)}\eqd \pp{\pp{\frac{\wt\Gamma_j^{-1/\alpha}}{\sif k1\wt\Gamma_k^{-1/\alpha}}}_{j\in\N},\Gamma(1-\alpha)\pp{\sif k1\wt\Gamma_k^{-1/\alpha}}^{-\alpha}},
\]
since on both sides the second random variable is the same deterministic transform of the first random sequence. It follows that 
\[
\pp{ \pp{\frac{S_{\alpha}}{\Gamma(1-\alpha)}}^{-1/\alpha}P_j^\downarrow}_{j\in\N}\eqd \pp{\wt\Gamma_j^{-1/\alpha}}_{j\in\N},
\]
as claimed.

Part (ii) is well-known \citep[Chapter 3.3]{pitman06combinatorial}, for which we briefly recall the density formula. For general $\theta>-\alpha$, suppose  $(P_j^\downarrow)_{j\in\N}$ has law $\mathsf P_{\alpha,0}$, and let $\proba$ denote the probability measure on the probability space. Then, on the same probability space, consider  the measure $\mathbb Q_{\alpha,\theta}$ determined by the following change of measure
\equh\label{eq:Q}
\frac{\d \mathbb Q_{\alpha,\theta}}{\d \proba}(\omega) = \frac{\Gamma(\theta+1)}{\Gamma(\theta/\alpha+1)}S_{\alpha}^{\theta/\alpha}.
\eque
It is known that $\mathbb Q_{\alpha,\theta}$ is a probability measure and  $(P_j^\downarrow)_{j\in\N}$ under $\mathbb Q_{\alpha,\theta}$ has the law $\mathsf P_{\alpha,\theta}$.
\end{proof}

\section{A quenched functional central limit theorem}\label{sec:proof}

We shall prove a stronger quenched functional central limit theorem for $(W_n,Y_n)_{n\in\N}$ with $W_n = (W_n(t))_{t\in[0,1]}$ and $Y_n = (Y_n(t))_{t\in[0,1]}$, from which Theorem \ref{thm:1} follows immediately.  For this purpose,  we shall rely on the representation of asymptotic frequencies $(P_j^\downarrow)_{j\in\N}$ developed in Lemma \ref{lem:P_j}. Set random variables
\[
\Gamma_j:=\frac{S_\alpha}{\Gamma(1-\alpha)}\pp{P_j^\downarrow}^{-\alpha}, j\in\N,
\]
and
\equh\label{eq:N(t)}
N(t):=\max\{n\in\N: \Gamma_n\le t\}, t\ge 0,
\eque
and $\max\emptyset=0$ by convention. Set, for a sequence of decreasing numbers $(\epsilon_n)_{n\in\N}, \epsilon_n\downarrow 0$ as $n\to\infty$,  
\[
\calP_n:=\sigma(N(t):t\in[0, \epsilon_n n^\alpha]).
\]
Clearly, $\calP_n\subset\sigma((P_j^\downarrow)_{j\in\N})\subset\calP$. Lemma \ref{lem:P_j} shows that when $\theta=0$, $(N(t))_{t\ge 0}$ in \eqref{eq:N(t)} is a standard Poisson process with $(\Gamma_j)_{j\in\N}$ its consecutive arrival times, while for other $\theta>-\alpha$ this representation does not hold directly, although the corresponding law is obtained by a change of measure. 

We shall then prove the following quenched version of Theorem \ref{thm:1}. 
\begin{theorem}\label{thm:2}
Let $W_n$ and $Y_n$ be as in Theorem \ref{thm:1}, and assume that 
\equh\label{eq:epsilon_n}
\epsilon_nn^{\alpha}\uparrow\infty \qmand \epsilon_n\log\log n\to 0,
\eque
as $n\to\infty$. We have
\[
\calL\pp{\pp{(W_n(t))_{t\in[0,1]},(Y_n(t))_{t\in[0,1]}}\mmid\calP_n}\aswto\calL\pp{\pp{S_\alpha^{1/2}(Z_\alpha\topp1(t))_{t\in[0,1]},S_\alpha^{1/2}(Z_\alpha\topp2(t))_{t\in[0,1]}}\mmid\calP}
\]
in $D[0,1]^2$ as $n\to\infty$, where $Z_\alpha\topp1,Z_\alpha\topp2$ are independent from $\calP$.
\end{theorem}
Here, for almost sure weak convergence our reference is \citet{grubel16functional}. For random elements $(X_n)_{n\in\N}$ and $X$ in a complete and separable metric space $M$, we write $\calL(X_n\mid\calP_n)\aswto \calL(X\mid\calP)$ as $n\to\infty$, if $\limn \esp (f(X_n)\mid\calP_n) = \esp(f(X)\mid\calP)$ almost surely for all continuous and bounded functions $f:M\to\R$. Again,  in order to define the conditional expectations we assume implicitly that $(X_n)_{n\in\N},X$ are defined on a common probability space of which $(\calP_n)_{n\in\N},\calP$ are $\sigma$-algebras.  When $\calP_n\equiv \calP$, we simply write $X_n\aswto X$ with respect to $\calP$ as $n\to\infty$ as in \eqref{eq:DW16}. 

In particular, throughout we assume $S_\alpha$ is defined in \eqref{eq:diversity} and is $\calP$-measurable. We do not repeat this in the statements of quenched limit theorems in the sequel.

We have already established a quenched convergence of the process $W_n$ \eqref{eq:DW16} \citep{durieu16infinite}. It turns out that to prove Theorem \ref{thm:2} it suffices to establish the following quenched convergence for $Y_n$ and \eqref{eq:DW16}. 
\begin{proposition} \label{prop:CLT Y}
For $Y_n$ in \eqref{eq:Y} with  $\alpha\in(0,1),\theta>-\alpha$, under \eqref{eq:epsilon_n},
\[
\calL\pp{\pp{Y_n(t)}_{t\in[0,1]}\mmid\calP_n}\aswto\calL\pp{ S_{\alpha}^{1/2}\pp{Z_\alpha\topp2(t)}_{t\in[0,1]}\mmid\calP},
\]
as $n\to\infty$ in $D[0,1]$, where $Z_\alpha\topp2$ is  independent from $\calP$.
\end{proposition}
Once the above proposition is established, the proof of Theorem \ref{thm:2} is relatively simple, and is provided in Section \ref{sec:WY} at the end. A similar idea has already been applied in \citet{durieu16infinite}. 
Moreover, for each $\alpha\in(0,1)$ fixed, the above results for $\theta>-\alpha, \theta\ne 0$ follow relatively easily from the case $\theta=0$, as explained below. (It is key to have the {\em quenched} convergence of $Y_n$ with $\theta=0$; if only the {\em annealed} convergence is established then the argument below does not work.) 
\begin{proof}[Proof of Proposition \ref{prop:CLT Y} with $\theta>-\alpha,\theta\ne0$]
{\em Assume that Proposition \ref{prop:CLT Y} has been proved with $\theta=0$.} 
Recall that the law of $(\alpha,\theta)$-partitions can be derived from the law of $(\alpha,0)$-partitions via a change of measure as discussed around \eqref{eq:Q}. 
Note also that for all bounded random variables $X$,
\[
\esp_{\mathbb Q_{\alpha,\theta}}\pp{X\mid\calP_n} = \frac{\esp(Q_{\alpha,\theta}X\mid\calP_n)}{\esp(Q_{\alpha,\theta}\mid\calP_n)}.
\]
Then, to prove the claimed result it suffices to show that, always assuming below that $Y_n$ is based on $(\alpha,0)$-partitions and setting $Y:=S_\alpha^{1/2}(Z_\alpha\topp2(t))_{t\in[0,1]}$, 
\equh\label{eq:theta ne 0}
\limn \esp\pp{Q_{\alpha,\theta} f(Y_n)\mmid\calP_n} = \esp \pp{Q_{\alpha,\theta}f(Y)\mmid\calP}
\eque
where $f$ is a continuous and bounded function and 
\[ Q_{\alpha,\theta} := \frac{\Gamma(\theta+1)}{\Gamma(\theta/\alpha+1)}S_\alpha^{\theta/\alpha} = \frac{\Gamma(\theta+1)}{\Gamma(\theta/\alpha+1)}\pp{\Gamma(1-\alpha)\pp{\sif j1 \Gamma_j^{-1/\alpha}}^{-\alpha}}^{\theta/\alpha}.
\]
(Indeed, \eqref{eq:theta ne 0} implies that $\limn \esp(Q_{\alpha,\theta}\mid\calP_n) = Q_{\alpha,\theta}>0$ almost surely, and hence it suffices to examine the convergence of the corresponding numerators.)

Write $Q=Q_{\alpha,\theta}$ from now on, and set
\[
Q_n:=
\begin{cases}
\displaystyle
\frac{\Gamma(\theta+1)}{\Gamma(\theta/\alpha+1)}
\pp{\Gamma(1-\alpha)\pp{\sum_{j:\Gamma_j\leq\epsilon_n n^\alpha}\Gamma_j^{-1/\alpha}}^{-\alpha}}^{\theta/\alpha},
&\mbox{if }N(\epsilon_n n^\alpha)\geq1,\\
1,&\mbox{otherwise.}
\end{cases}
\]

Notice that $Q_n$ is $\calP_n$-measurable. Thus, to show \eqref{eq:theta ne 0} it suffices to remark
\begin{align*}
\esp(Qf(Y_n)\mid\calP_n) - \esp (Qf(Y)\mid\calP) 
& = 
\esp(Qf(Y_n)\mid\calP_n) - \esp (Q_nf(Y_n)\mid\calP_n)\\
& \quad +
Q_n\esp(f(Y_n)\mid\calP_n) - Q_n\esp (f(Y)\mid\calP)\\
& \quad +
Q_n\esp(f(Y)\mid\calP) - Q\esp (f(Y)\mid\calP),
\end{align*}
and each of the three differences on the right-hand side converges to zero almost surely.

Indeed, for the first difference, we have
\[
\abs{\esp\pp{(Q-Q_n)f(Y_n)\mmid\calP_n}}
\leq\nn f_\infty\esp\pp{|Q-Q_n|\mmid\calP_n}\to0
\]
almost surely by the conditional dominated convergence \citep[Theorem 5.5.9]{durrett10probability}.
 Indeed, $Q_n\to Q$ almost surely, and $|Q-Q_n|$ is bounded by an integrable random variable. For $\theta>0$, writing
\[
c_{\alpha,\theta}:=
\frac{\Gamma(\theta+1)}{\Gamma(\theta/\alpha+1)}
\Gamma(1-\alpha)^{\theta/\alpha},
\]
we have $|Q-Q_n|\leq Q+1+c_{\alpha,\theta}\Gamma_1^{\theta/\alpha}$; for $\theta<0$, we have $|Q-Q_n|\leq Q+1$. 
Then, the second difference converges to zero by Proposition \ref{prop:CLT Y} with $\theta=0$, and the third difference converges to zero since $Q_n\to Q$ almost surely.
\end{proof}
The majority of the rest of this section is devoted to the proof of Proposition \ref{prop:CLT Y} with $\theta = 0$.

From now on, we consider equivalently the total number of non-empty urns in a Karlin model with random frequencies $(P_j)_{j\in\N}$ (i.e.~the nonparametric Bayesian view). That is, given $\calP$, we sample conditionally i.i.d.~random variables $(X_j)_{j\in\N}$ with $\proba(X_j = \ell\mid\calP) = P_\ell^\downarrow, \ell\in\N$, and set
\[
K_{n,\ell} := \summ j1n \inddd{X_j = \ell}, \ell\in\N \qmand K_n:=\sif \ell1 \inddd{K_{n,\ell}>0}. 
\]
The $K_n$ defined this way has the same law as the number of blocks in the $(\alpha,\theta)$-partition of $\{1,\dots,n\}$.
For computational convenience, the idea of Poissonization is to consider the following approximation. Given $\calP$, let $(\Lambda_\ell(t))_{t\ge0}, \ell\in\N$ be conditionally independent Poisson processes with respective rates $P_\ell^\downarrow$. The approximation of $K_n$ is
\[
\wt K(t):=\sif \ell1 \inddd{\Lambda_\ell(t)>0}, t\ge 0.
\]
The advantage of working with $\wt K(t)$ is that given $\calP$ it is a summation of conditionally independent random variables. It is clear that limit theorems concerning $\wt K(t)$ are easier to prove thanks to independence than those concerning $K_n$, and this fact has been exploited substantially in the analysis of $W_n$ \citep{karlin67central,durieu16infinite}.
In the analysis of $Y_n$, however, we do not exploit this independence, but instead use a continuous mapping argument on functionals of a Poisson process. The obstacle here is that the involved Poisson process takes the form of $(N(Dt))_{t\ge 0}$ where the random variable $D$ depends on the Poisson process itself. Most of the effort is devoted to  decoupling the dependence between the two as $t\to\infty$ (Lemma \ref{lem:decouple ND} below). We then need to control the difference between $\wt K(n)$ and $K_n$. This step is known as the de-Poissonization: often a {\em functional} central limit theorem for $\wt K$ is established, and the tightness for the functional convergence might be quite challenging to prove.

\subsection{Proof for the Poissonized model}
Set $\nu(t) := \max\{j: P_j^\downarrow\ge 1/t\}$. 	
Then,
\[
\esp \pp{\wt K(n)\mmid\calP} = \int_0^\infty(1-e^{-n/x})\nu(\d x) = \int_0^\infty \frac n{x^2}e^{-n/x}\nu(x)\d x = \int_0^\infty e^{-x}\nu(n/x)\d x.
\]
Recall $\theta=0$ and the representation of $P_j^\downarrow$ in Lemma \ref{lem:P_j}. 
For notational convenience, introduce
\[
D := \frac{S_\alpha}{\Gamma(1-\alpha)} = \pp{\sif j1 \Gamma_j^{-1/\alpha}}^{-\alpha}.
\]
In particular, 
$P_j^\downarrow = D^{1/\alpha} \Gamma_j^{-1/\alpha}$,
and hence
\[
\nu(t) = \max\ccbb{j\in\N:P_j^\downarrow\ge \frac 1t} = N\pp{D t^\alpha}.
\]
	Notice also that 
	\[
	\int_0^\infty e^{-x}\pp{\frac tx}^\alpha \d x = t^\alpha\Gamma(1-\alpha).
	\]
	The statistic of interest is now
\[
	\wt Y_n(t):=\frac{\esp\spp{\wt K(nt)\mid\calP} - \Gamma(1-\alpha)D (nt)^\alpha}{n^{\alpha/2}} = \int_0^\infty e^{-x}\frac{N\pp{D(nt/x)^\alpha}- D(nt/x)^\alpha}{n^{\alpha/2}}\d x.
\]
Note that the integrand above is integrable at both $0$ and $\infty$ almost surely.

The difficulty of the analysis is that the Poisson process $N$ induced by $(\Gamma_n)_{n\in\N}$ and $D$ are dependent. 
We start with a heuristic calculation that identifies the limit process.
By the functional central limit theorem of a Poisson process, we know that
	\[
	\pp{\frac{N(nt)-nt}{\sqrt n}}_{t\in[0,\infty)}\weakto (\B_t)_{t\in[0,\infty)}
	\]
	in $D[0,\infty)$ where $\B$ is a standard Brownian motion. 
	Thus, {\em assuming} $D$ is independent from $N(t)$ (which is not true here so we are using a little abuse of notation), one {\em expects}
	\begin{align*}
		(\wt Y_n(t))_{t\in[0,1]}&\weakto \pp{\int_0^\infty e^{-x}\B_{D(t/x)^\alpha}\d x}_{t\in[0,1]} \eqd D^{1/2}\pp{\int_0^\infty e^{-x}\B_{(t/x)^\alpha}\d x}_{t\in[0,1]}\\
		&= \pp{\frac{S_\alpha}{\Gamma(1-\alpha)}}^{1/2}\pp{\int_0^\infty (1-e^{-t/x^{1/\alpha}})\d \B_x}_{t\in[0,1]}\eqd S_\alpha^{1/2}\pp{Z
		\topp2_\alpha(t)}_{t\in[0,1]},
	\end{align*}
	where $Z_\alpha\topp2$ is a centered Gaussian process with \[
	\cov(Z\topp2_\alpha(s),Z\topp2_\alpha(t)) = s^\alpha+t^\alpha-(s+t)^\alpha, s,t>0,
	\] 
	as introduced earlier.
The first equality in distribution follows from self-similarity of Brownian motion and independence between $D$ and $\B$. The equality step follows from stochastic integration by parts. The last step follows from It\^o's isometry: we have
\begin{align*}
\int_0^\infty(1-e^{-s/x^{1/\alpha}})(1-e^{-t/x^{1/\alpha}})\d x & = \int_0^\infty (1-e^{-sy})(1-e^{-ty})\alpha y^{-1-\alpha}\d y\\
& = \int_0^\infty\pp{1-e^{-sy}-e^{-ty}+e^{-(s+t)y}}\alpha y^{-\alpha-1}\d y \\
&= \Gamma(1-\alpha)\pp{s^\alpha+t^\alpha-(s+t)^\alpha}.
\end{align*}

To make the above argument rigorous, the key step is the following.
\begin{proposition} \label{prop:CLT Poissonized}
With $(\epsilon_n)_{n\in\N}$ satisfying \eqref{eq:epsilon_n}, $\alpha\in(0,1)$, and $\theta=0$,
		\[
\calL\pp{\pp{\wt Y_n(t)}_{t\in[0,1]}\mmid\calP_n}\aswto\calL\pp{ \pp{S_{\alpha}^{1/2}Z_\alpha\topp2(t)}_{t\in[0,1]}\mmid\calP}
		\]
		as $n\to\infty$ in $D[0,1]$, where $Z_\alpha\topp2$ is independent from $\calP$.
	\end{proposition}
	
The first step is to prove the following. 
\begin{lemma}\label{lem:decouple ND}With $(\epsilon_n)_{n\in\N}$ satisfying \eqref{eq:epsilon_n}, $\alpha\in(0,1)$, and $\theta=0$,
\[
\calL\pp{\pp{\frac{N(D(nt)^\alpha) - D(nt)^\alpha}{n^{\alpha/2}}}_{t\in[0,\infty)}\mmid\calP_n}\aswto\calL\pp{ D^{1/2} (\B_{t^{\alpha}})_{t\in[0,\infty)}\mmid\calP},\]
in $D[0,\infty)$,
		where on the right-hand side $D = \Gamma(1-\alpha)\inv S_\alpha$ and the Brownian motion~$\B$ is independent from $\calP$.
	\end{lemma}
	\begin{proof}It suffices to prove the convergence in the functional space $D[0,K]$ for any fixed $K$  \citep{billingsley99convergence}, and for the sake of simplicity we consider $K=1$.

Introduce
\equh\label{eq:D_n}
D_{n} := 
\begin{cases}
\displaystyle
\pp{\sum_{j:\Gamma_j\le \epsilon_n n^\alpha}\Gamma_j^{-1/\alpha}}^{-\alpha}, & \mbox{ if } N(\epsilon_n n^\alpha)\ge 1, \\
1, &\mbox{ otherwise.}
\end{cases}\eque
So $D_n\to D$ almost surely (here we need $\epsilon_nn^\alpha\to\infty$). The value 1 in the second case is irrelevant and we just need $D_n$ to be a well-defined finite number. Note also that $D_n$ is $\calP_n$-measurable.	To simplify the notation, write
\[
\wb N_n(t) := \frac{N(n^\alpha t)-n^\alpha t}{n^{\alpha/2}}.
\]
	So
\[
	\wt Y_n(t) = \int_0^\infty e^{-x}\wb N_n(D(t/x)^\alpha)\d x.
\]

We start by writing
		\begin{align}
			\wb N_n(Dt^\alpha) &= \wb N_n(D_nt^\alpha+\epsilon_n)- \wb N_n(\epsilon_n)\nonumber\\
			& \quad + \wb N_n(Dt^\alpha) - \wb N_n(D_n t^\alpha) + \wb N_n(D_nt^\alpha) - \wb N_n(D_nt^\alpha+\epsilon_n) + \wb N_n(\epsilon_n)\nonumber\\
			& =: \wb N_n(D_nt^\alpha+\epsilon_n) - \wb N_n(\epsilon_n) + R_n(t^\alpha).\label{eq:decomp}
		\end{align}
Set
\[
\pp{\what N_n(t)}_{t\ge 0} := \pp{\wb N_n(\epsilon_n+t) - \wb N_n(\epsilon_n)}_{t\ge 0}.
\]
By the independent-increment property, $\what N_n$ is a normalized standard Poisson process independent from $\calP_n$, and hence from $D_n$. Thus, the decomposition \eqref{eq:decomp} becomes
\[
\wb N_n(Dt^\alpha) = \what N_n(D_nt^\alpha) + R_n(t^\alpha), t\ge 0.
\]
Since $D_n$ is $\calP_n$-measurable and $D_n\to D$ almost surely, the functional central limit theorem and Slutsky's lemma applied pointwise in $\omega$ yield
		\[
		\calL\pp{\pp{\what N_n(D_n t^\alpha)}_{t\in[0,1]} \mmid \calP_n}\aswto \calL(D^{1/2}(\B_{t^{\alpha}})_{t\in[0,1]}\mid\calP),
		\]
as $n\to\infty$. 
Therefore, to conclude the proof it remains to prove the following:
\equh\label{eq:uniform control}
\limsupn \proba\pp{\sup_{t\in[0,1]}|R_n(t^\alpha)|>\eta\mmid\calP_n} = 0, \mfa \eta>0.
\eque

Set
\begin{align*}
\Delta_n& :=\sup_{u\in[0,\epsilon_n]}\sabs{\wb N_n(u)},\\
 \what N_n^*(t) &:= \sup_{u\in[0,t]}\sabs{\what N_n(u)},\\
  \omega_{\what N_n}(t,\delta) &:= \sup_{u,v\in[0,t],|u-v|\le \delta}\sabs{\what N_n(u)-\what N_n(v)}.
\end{align*}
Note that $\Delta_n$ is $\calP_n$ measurable, and both $\what N_n^*(t)$ and $\omega_{\what N_n}(t,\delta)$ are independent from $\calP_n$. 
Recall the law of the iterated logarithm for a standard Poisson process $\limsupn |N(n)-n|/\sqrt{2n\log\log n} = 1$. 
Notice that $\Delta_n = O(\sqrt{\epsilon_n \log\log (\epsilon_n n^\alpha)})$ almost surely. It then follows that $\Delta_n\to 0$ almost surely under \eqref{eq:epsilon_n}.  We also have $\what N_n^*(t)\weakto \sup_{s\in[0,t]}|\B_s|$ by the continuous mapping theorem. Note that $\proba(\omega_{\what N_n}(t,\delta)>\eta\mid\calP_n) = \proba(\omega_{\wb N_n}(t,\delta)>\eta)$ (with $\omega_{\wb N_n}$ defined similarly as the modulus of continuity of the process $\wb N_n$) and 
\equh\label{eq:Donsker Poisson}
\lim_{\delta\downarrow 0}\limsupn\proba\pp{\omega_{\wb N_n}(t,\delta)>\eta} = 0, \mfa \eta>0.
\eque
By the functional central limit theorem for the Poisson process, $(\wb N_n(s))_{s\in[0,t]}$ converges weakly to $(\B_s)_{s\in[0,t]}$. Since the limiting process has continuous sample paths, \eqref{eq:Donsker Poisson} follows.

For $\delta>0$, set
\[
A_{n,\delta}:=
\left\{
N\pp{\epsilon_n n^\alpha}\geq1,
D_n-D<\delta,
\epsilon_n<\min\{\delta,\Gamma_1\}
\right\}.
\]
Note that on the event $A_{n,\delta}$, $D_n\le \Gamma_1$. 
Since $D_n\to D<\Gamma_1$ almost surely and $\epsilon_n\to0$, we have
\[
\inddd{A_{n,\delta}^c}\to0
\qquad\text{almost surely,}
\]
for every fixed $\delta>0$. The estimates below show that
\begin{equation}\label{eq:Rn bound}
\sup_{t\in[0,1]}|R_n(t^\alpha)|
\leq5\left(\Delta_n+\omega_{\what N_n}(\Gamma_1,\delta)
+\what N_n^*(\delta)\right) \quad \mbox{ on
$A_{n,\delta}$}.
\end{equation}
Consequently,
\begin{multline}\label{eq:4 terms}
\proba\left(\sup_{t\in[0,1]}|R_n(t^\alpha)|>\eta\middle|\calP_n\right)\leq
\proba(A_{n,\delta}^c\mid\calP_n)
+\proba\left(\Delta_n>\frac{\eta}{15}\middle|\calP_n\right)\\
+\proba\left(\omega_{\what N_n}(\Gamma_1,\delta)>
\frac{\eta}{15}\middle|\calP_n\right)
+\proba\left(\what N_n^*(\delta)>
\frac{\eta}{15}\middle|\calP_n\right).
\end{multline}
Here and in the sequel, we shall need the conditional dominated convergence \citep[Theorem 5.5.9]{durrett10probability} a few times: if $X_n\to X$ almost surely, $|X_n|\le Z$ for all $n$ with $\esp Z<\infty$, and $\calF_n\uparrow\calF_\infty$, then $\esp(X_n\mid\calF_n)\to\esp(X\mid\calF_\infty)$ almost surely. 

The first term on the right-hand side of \eqref{eq:4 terms} goes to zero almost surely by the conditional dominated convergence
for every fixed $\delta>0$. 
The second term equals
$\inddd{\Delta_n>\eta/15}$ and hence also converges to zero almost surely.
Moreover, almost surely, $\Gamma_1$ is $\calP_n$-measurable for all
sufficiently large $n$. Therefore, by the independence of $\what N_n$ from
$\calP_n$ and the functional central limit theorem,
\begin{align*}
\lim_{\delta\downarrow0}\limsup_{n\to\infty}
\pp{\proba\left(\omega_{\what N_n}(\Gamma_1,\delta)>
\frac{\eta}{15}\middle|\calP_n\right)+\proba\left(\what N_n^*(\delta)>
\frac{\eta}{15}\middle|\calP_n\right)}=0
\qquad\text{almost surely}.
\end{align*}
This proves \eqref{eq:uniform control}.

It remains to prove \eqref{eq:Rn bound}. On the event $A_{n,\delta}$,
$D_n\le\Gamma_1$, $D_n-D<\delta$, and
$\epsilon_n<\min\{\delta,\Gamma_1\}$.
\begin{align*}
\sup_{t\in[0,1]}\sabs{\wb N_n(Dt^\alpha) - \wb N_n (D_nt^\alpha)} &\le \sup_{u,v\in[0,\Gamma_1], |u-v|\le D_n-D}\sabs{\wb N_n(u) - \wb N_n(v)}\\
& \le \sup_{u,v\in[0,\epsilon_n+\delta]}|\wb N_n(u)-\wb N_n(v)| \vee \sup_{\substack{u,v\in[\epsilon_n,\Gamma_1]\\|u-v|\le\delta}}\sabs{\wb N_n(u)-\wb N_n(v)}\\
&\le 2\sup_{u\in[0,\epsilon_n+\delta]}\sabs{\wb N_n(u)} \vee \sup_{u,v\in[\epsilon_n,\Gamma_1], |u-v|\le\delta}\sabs{\wb N_n(u)-\wb N_n(v)}.
\end{align*}
Notice that
\[
 \sup_{u,v\in[\epsilon_n,\Gamma_1], |u-v|\le\delta}\sabs{\wb N_n(u)-\wb N_n(v)}= \sup_{u,v\in[0,\Gamma_1-\epsilon_n],|u-v|\le\delta}\sabs{\what N_n(u)-\what N_n(v)} \le \omega_{\what N_n}(\Gamma_1,\delta),
 \]
 and
\equh
 \sup_{u\in[0,\epsilon_n+\delta]}\sabs{\wb N_n(u)}\le \sup_{u\in[0,\epsilon_n]}\sabs{\wb N_n(u)}\vee \sup_{u\in[0,\delta]}\pp{\sabs{\what N_n(u)}+\sabs{\wb N_n(\epsilon_n)}} \le \Delta_n+\what N_n^*(\delta). \label{eq:sup bound}
\eque
 Therefore, 
\equh\label{eq:A}
\sup_{t\in[0,1]}\sabs{\wb N_n(Dt^\alpha) - \wb N_n (D_nt^\alpha)} \le 2\Delta_n+2\what N_n^*(\delta)+\omega_{\what N_n}(\Gamma_1,\delta).
\eque
Next,
\begin{align}
\sup_{t\in[0,1]}&\sabs{\wb N_n(D_n t^\alpha) - \wb N_n(D_nt^\alpha+\epsilon_n)} \le
\sup_{\substack{u,v\in[0,\Gamma_1+\epsilon_n]\\ |u-v|=\epsilon_n}}\sabs{\wb N_n(u)-\wb N_n(v)} \nonumber
\\
& \le \sup_{\substack{u,v\in[0,2\epsilon_n]\\|u-v|=\epsilon_n}}\sabs{\wb N_n(u)-\wb N_n(v)} \vee \sup_{\substack{u,v\in[\epsilon_n,\Gamma_1+\epsilon_n]\\|u-v|=\epsilon_n}}\sabs{\wb N_n(u)-\wb N_n(v)}\nonumber\\
& \le 2\sup_{u\in[0,2\epsilon_n]}\sabs{\wb N_n(u)} + \omega_{\what N_n}(\Gamma_1,\epsilon_n) \le 2\Delta_n + 2\what N_n^*(\epsilon_n)+\omega_{\what N_n}(\Gamma_1,\epsilon_n).\label{eq:B}
\end{align}
We have used \eqref{eq:sup bound} in the last inequality again. Combining \eqref{eq:sup bound}, \eqref{eq:A}, \eqref{eq:B}, and the definition of $R_n(t^\alpha)$ in \eqref{eq:decomp}, we have thus proved \eqref{eq:Rn bound}.
\end{proof}


	\begin{proof}[Proof of Proposition \ref{prop:CLT Poissonized}]
We first prove the convergence of finite-dimensional distributions $(\wt Y_n(t))_{t\in [0,1]}$; that is, for all $k\in\N, t_1,\dots,t_k\in[0,1]$,
\equh\label{eq:fdd}
\calL\pp{\pp{\wt Y_n(t_j)}_{j=1,\dots,k}\mmid\calP_n}\aswto \calL\pp{\pp{D^{1/2}\int_0^\infty e^{-x}\B_{(t_j/x)^\alpha}\d x}_{j=1,\dots,k}\mmid\calP},
\eque
as $n\to\infty$.
We continue to use the notations introduced in Lemma \ref{lem:decouple ND}.
For each $\epsilon\in(0,1)$ (the choice is independent from $(\epsilon_n)_{n\in\N}$), we  decompose $\wt Y_n(t)$ into
\equh
			\wt Y_n(t) = \wt Y_{n,\epsilon}\topp1(t) + \wt Y\topp2_{n,\epsilon}(t),
			 \label{eq:decom In} 
\eque
with
       \begin{align*}
            \wt Y\topp1_{n,\epsilon}(t):= \int_0^{\epsilon} e^{-x} \wb N_n(D(t/x)^\alpha) \, \d x, \qmand
            \wt Y\topp2_{n,\epsilon}(t):= \int_{\epsilon}^\infty e^{-x} \wb N_n(D(t/x)^\alpha) \, \d x. 
        \end{align*} 
By Lemma \ref{lem:decouple ND}, it follows that
\[
\calL\pp{\pp{\wt Y\topp2_{n,\epsilon}(t_j)}_{j=1,\dots,k}\mmid\calP_n}\aswto \calL\pp{D^{1/2}\pp{\int_\epsilon^\infty e^{-x}\B_{(t_j/x)^\alpha}\d x}_{j=1,\dots,k}\mmid\calP},
\]		
by the continuous mapping theorem. It is clear that the right-hand side converges to the claimed limit as $\epsilon\downarrow 0$.
Then, the claimed convergence follows by a triangular array argument  (e.g.~\citet[Theorem 3.2]{billingsley99convergence}) if we could prove
\[
 \lim\limits_{\epsilon\downarrow 0} \limsup\limits_{n \rightarrow \infty} \mathbb{P} \left(|\wt Y^{(1)}_{n,\epsilon}(t)| \ge \eta \mmid\calP_n\right) = 0, \mbox{ almost surely, } \quad \mfa\eta>0. 
\]
We continue to rely on the decomposition involving $\epsilon_n$ and estimates on $\Delta_n,\what N_n$ in the proof of Lemma \ref{lem:decouple ND}.
Indeed, by \eqref{eq:sup bound} (inequalities below hold for all sufficiently large $n$),
\begin{align*}
\sup_{s\in[0,t]}\sabs{\wb N_n(Ds^\alpha)} & \le \sup_{u\in[0,\Gamma_1t^\alpha]}\sabs{\wb N_n(u)}\le \Delta_n +\what N_n^*(\Gamma_1t^\alpha).
\end{align*}
Thus,
\[
\abs{\int_0^\epsilon e^{-x}\wb N_n(D(t/x)^\alpha)\d x} \le \epsilon \Delta_n+\int_0^\epsilon e^{-x} \what N_n^*(\Gamma_1(t/x)^\alpha)\d x.
\]
Therefore,
\[
\proba\pp{\sabs{\wt Y_{n,\epsilon}\topp1(t)}>\eta\mmid\calP_n}\le \proba\pp{\epsilon\Delta_n+\int_0^\epsilon e^{-x}\what N_n^*(\Gamma_1(t/x)^\alpha)\d x>\eta\mmid\calP_n}.
\]
Since $\Delta_n\to 0$ almost surely, it suffices to control, writing $\esp_{\calP_n}(\cdot) = \esp(\cdot\mid\calP_n)$, 
\begin{align*}
\proba\pp{\int_0^\epsilon e^{-x}\what N_n^*(\Gamma_1(t/x)^\alpha)\d x>\eta\mmid\calP_n} & \le \frac1{\eta^2}\esp\pp{\pp{\int_0^\epsilon e^{-x}\what N_n^*(\Gamma_1(t/x)^\alpha)\d x}^2\mmid\calP_n}\\
& \le \frac \epsilon{\eta^2}\int_0^\epsilon \esp\pp{\what N_n^*(\Gamma_1(t/x)^\alpha)^2\mmid\calP_n}\d x.
\end{align*}
Recall the definition of $\what N_n^*$. By Doob's martingale inequality
\[
\esp \what N_n^*(t)^2 \le 4\esp \what N_n(t)^2 = 4 \esp \pp{\frac{N(tn^{\alpha})-tn^\alpha}{n^{\alpha/2}}}^2 = 4t,
\]
whence $\esp(\what N_n^*(\Gamma_1(t/x)^\alpha)^2\mid\calP_n)\le 4\Gamma_1 (t/x)^\alpha$ almost surely. That is, for all $\eta>0$,
\equh\label{eq:integral s_n*}
\proba\pp{\int_0^\epsilon e^{-x}\what N_n^*(\Gamma_1(t/x)^\alpha)\d x>\eta\mmid\calP_n} \le \epsilon^{2-\alpha}\frac{4\Gamma_1 t^\alpha}{\eta^2(1-\alpha)}.
\eque
We have thus proved \eqref{eq:fdd}.

It now remains to show the tightness for $(\wt Y_n(t))_{t\in [0,1]}$. Recall the decomposition of $\wt Y_n(t)$ in \eqref{eq:decom In}. The tightness will then follow from the following assertions
\begin{align}
 \lim\limits_{\delta\downarrow 0} \limsup\limits_{n \rightarrow \infty} \mathbb{P} \left( \sup_{\substack{s,t \in [0,1],|s-t|\le \delta}} |\wt Y^{(1)}_{n,\epsilon}(t) - \wt Y^{(1)}_{n,\epsilon}(s)| > \eta \mmid\calP_n\right) &\le C\Gamma_1\epsilon^{2-\alpha},\label{eq:B.1} \\
\lim\limits_{\delta\downarrow 0} \limsup\limits_{n \rightarrow \infty} \mathbb{P} \left( \sup_{\substack{s,t \in [0,1],|s-t|\le \delta}} |\wt Y^{(2)}_{n,\epsilon}(t) - \wt Y^{(2)}_{n,\epsilon}(s)| > \eta \mmid\calP_n\right)& = 0,  \label{eq:B.2}
\end{align}
for all $\epsilon>0, \eta>0$.
This time,
\begin{align*}
\sup_{\substack{s,t\in[0,1]\\|s-t|\le \delta}}\sabs{\wt Y_{n,\epsilon}\topp1(t)-\wt Y_{n,\epsilon}\topp1(s)}& \le
\sup_{\substack{s,t\in[0,1]\\|s-t|\le\delta}}\int_0^\epsilon e^{-x}\sabs{\wb N_n(D(t/x)^\alpha)- \wb N_n(D(s/x)^\alpha)}\d x\nonumber\\
 & \le 2\int_0^\epsilon e^{-x}\sup_{t\in[0,1]}\sabs{\wb N_n(D(t/x)^\alpha)}\d x \le 2\epsilon\Delta_n+2\int_0^\epsilon \what N_n^*(\Gamma_1 x^{-\alpha})\d x,
\end{align*}
which yields \eqref{eq:B.1} (by a similar argument around \eqref{eq:integral s_n*}).
Next,
\begin{align*}
\sup_{x\ge \epsilon}\sup_{\substack{s,t\in[0,1]\\|s-t|\le\delta}}&\sabs{\wb N_n(D(t/x)^\alpha) - \wb N_n(D(s/x)^\alpha)}
 \le \sup_{x\ge\epsilon}\sup_{\substack{u,v\in[0,\Gamma_1 x^{-\alpha}]\\|u-v|\le \Gamma_1(\delta/x)^\alpha}}\sabs{\wb N_n(u)-\wb N_n(v)}\\
& \le \sup_{\substack{u,v\in[0,\epsilon_n+\Gamma_1(\delta/\epsilon)^\alpha]\\|u-v|\le \Gamma_1(\delta/\epsilon)^\alpha}}\sabs{\wb N_n(u)-\wb N_n(v)}\vee \sup_{\substack{u,v\in [\epsilon_n,\Gamma_1 \epsilon^{-\alpha}]\\|u-v|\le\Gamma_1(\delta/\epsilon)^\alpha}}\sabs{\wb N_n(u)-\wb N_n(v)}\\
& \le 2\Delta_n+2\what N_n^*(\Gamma_1(\delta/\epsilon)^\alpha) + \omega_{\what N_n}(\Gamma_1\epsilon^{-\alpha},\Gamma_1(\delta/\epsilon)^\alpha).
\end{align*}
That is,
\begin{multline*}
\sup_{\substack{s,t\in[0,1]\\|s-t|\le \delta}}\int_\epsilon^\infty e^{-x}\sabs{\wb N_n(D(t/x)^\alpha) - \wb N_n(D(s/x)^\alpha)}\d x\\
\le  2\Delta_n+2\what N_n^*(\Gamma_1(\delta/\epsilon)^\alpha) + \omega_{\what N_n}(\Gamma_1\epsilon^{-\alpha},\Gamma_1(\delta/\epsilon)^\alpha).
\end{multline*}
This time, 
\[
\proba\pp{2\what N_n^*(\Gamma_1(\delta/\epsilon)^\alpha)>\eta\mmid\calP_n}\le \frac{8\esp(\what N_n(\Gamma_1(\delta/\epsilon)^\alpha)^2\mid\calP_n)}{\eta^2}=\frac 8{\eta^2}\Gamma_1 (\delta/\epsilon)^\alpha. 
\]
Therefore taking limsup as $n\to\infty$ first and then $\delta\downarrow 0$ the corresponding iterated limit is zero for all $\epsilon,\eta>0$ fixed. Also, 
\[
\lim_{\delta\downarrow0}\limsupn\proba\pp{\omega_{\what N_n}(\Gamma_1\epsilon^{-\alpha},\Gamma_1(\delta/\epsilon)^\alpha)>\eta\mmid\calP_n} = 0, \mbox{ almost surely,}
\]
by standard estimates of modulus of continuity of a Poisson process. Now \eqref{eq:B.2} follows.  
\end{proof}
\subsection{De-Poissonization}
In this section, we transfer this result from the
Poissonized model $(\wt Y_n(t) )_{t\in[0,1]}$ to the original model $(Y_n(t))_{t\in[0,1]}$. 
This procedure is often referred to as  the de-Poissonization.
\begin{proof}[Proof of Proposition \ref{prop:CLT Y} with $\theta=0$]
Let us denote $\left( Y (t) \right)_{t\in[0,1]}:= S_\alpha^{1/2} \left( Z_\alpha\topp2(t) \right)_{t\in[0,1]}$.
Recall that we have proved the following convergence in Proposition \ref{prop:CLT Poissonized}
\[
\calL\pp{\left(\wt Y_n(t) \right)_{t\in[0,1]}\mmid\calP_n}\aswto
\calL\pp{
\pp{Y(t)}_{t\in[0,1]}\mmid\calP} \, 
\]
in $D[0,1]$. 
Our goal is to show that
\[
\calL\pp{\left(Y_n(t) \right)_{t\in[0,1]}\mmid\calP_n}
\aswto
\calL\pp{\left( Y (t) \right)_{t\in[0,1]}\mmid\calP}.
\]
Let us recall
\begin{align}
\wt Y_n(t) & = \frac{\esp(\wt K(nt)\mid\calP) - \Gamma(1-\alpha)D(nt)^\alpha}{n^{\alpha/2}},\nonumber\\
Y_n(t) & = \frac{\esp(K_{\floor{nt}}\mid\calP) - \Gamma(1-\alpha)D\floor{nt}^\alpha}{n^{\alpha/2}}.\label{eq:centering 1}
\end{align}
 Let $\rho_n(t):=\lfloor nt\rfloor/n$ and let $\mathbb I$ denote the identity function on $[0,1]$. Since $\sup_{t\in[0,1]}|\rho_n(t)-\mathbb I(t)|\le 1/n$  and $Y$ is in $C[0,1]$, by change-of-time lemma \citep[P.~151]{billingsley99convergence} it follows that 
\[
\calL\pp{\pp{\wt Y_n\circ\rho_n(t)}_{t\in[0,1]}\mmid\calP_n}\aswto\calL\pp{\pp{Y(t)}_{t\in[0,1]}\mmid\calP}.
\]

We next compare $\esp(\wt K(nt)\mid\calP)$ and $\esp (K_{\floor{nt}}\mid\calP)$. 
Set $p_j:=P_j^\downarrow$ and
\[
F(s):=\esp\pp{\wt K(s)\mmid\calP}=\sum_{j\geq1}(1-e^{-sp_j}),\qquad G_m:=\esp(K_m\mid\calP)=\sum_{j\geq1}\bigl(1-(1-p_j)^m\bigr).
\]
For $m\geq1$, the elementary estimate used in \citet[the proof of Lemma~4.2]{durieu16infinite} yields
\[
0\leq G_m-F(m)=\sum_{j\geq1}\left(e^{-mp_j}-(1-p_j)^m\right)\leq\frac{F(m)}{m};
\]
for $m=0$, both terms are zero. Then, for $\lfloor nt\rfloor\geq1$,
\[
0\leq Y_n(t)-\wt Y_n(\rho_n(t)) = \frac{G_{\floor{nt}}-F(\floor{nt})}{n^{\alpha/2}} \leq\frac{F(\lfloor nt\rfloor)}{\lfloor nt\rfloor n^{\alpha/2}},
\]
whereas the difference is zero when $\lfloor nt\rfloor=0$. Since $F$ is concave and $F(0)=0$, the function $s\mapsto F(s)/s$ is non-increasing on $(0,\infty)$. It follows that
\[
\sup_{t\in[0,1]}\left|Y_n(t)-\wt Y_n(\rho_n(t))\right|\leq\frac{F(1)}{n^{\alpha/2}}\to0
\]
almost surely. Thus, for every bounded function $h:D[0,1]\to\mathbb R$ that is $1$-Lipschitz with respect to $d_{J_1}\wedge1$,
\begin{align*}
&\left|\esp(h(Y_n)\mid\calP_n)-\esp(h(Y)\mid\calP)\right|\\
&\quad\leq \esp\pp{\abs{h(Y_n)-h\pp{\wt Y_n\circ\rho_n}}\mmid\calP_n}+\left|\esp\pp{h\pp{\wt Y_n\circ\rho_n}\mmid\calP_n}-\esp(h(Y)\mid\calP)\right|\\
&\quad\leq \frac{\esp(F(1)\mid\calP_n)}{n^{\alpha/2}}+\left|\esp\pp{h\pp{\wt Y_n\circ\rho_n}\mmid\calP_n}-\esp(h(Y)\mid\calP)\right|\to0
\end{align*}
almost surely. Note that $F(1) = \sif\ell1 (1-e^{-p_\ell})\le \sif\ell1p_\ell = 1$. Since bounded Lipschitz functions are convergence determining on $D[0,1]$, this proves the claimed convergence.
\end{proof}
\subsection{Joint convergence}\label{sec:WY}
With the quenched convergence of $W_n$ and the convergence of $Y_n$ established, we are ready to prove the main result. 
	\begin{proof}[Proof of Theorem \ref{thm:2}]	
Recall that we write $W_n = (W_n(t))_{t\in[0,1]}$ and $Y_n = (Y_n(t))_{t\in[0,1]}$. Write similarly $Z = Z_\alpha\topp1 = (Z_\alpha\topp1(t))_{t\in[0,1]}, Z' = Z_\alpha\topp2 = (Z_\alpha\topp2(t))_{t\in[0,1]}$, and furthermore $S \equiv S_\alpha^{1/2}$. Write also $\esp_\calP(\cdot) = \esp(\cdot\mid\calP)$. 
Throughout, $S_\alpha$ is defined as in \eqref{eq:diversity} on the same probability space as $W_n, Y_n$, and is $\calP$-measurable. In the sequel we assume further that $Z,Z'$ are on the same probability space, the two are mutually independent, and also independent from $\calP$.

The quenched convergence
$W_n \aswto S_\alpha^{1/2}Z_\alpha\topp1$
		with respect to $\calP$ as $n\to\infty$ (recall \eqref{eq:DW16}) is equivalent to
\[
\limn \esp_\calP f(W_n) =  \esp_\calP f\pp{SZ}, \mbox{ almost surely,}
\]
for all continuous and bounded functions $f:D[0,1]\to\R$.
Similarly, Proposition \ref{prop:CLT Y} is the same as 
\[
\limn \esp_{\calP_n} g(Y_n) = \esp_\calP g\pp{SZ'},
\]
for all continuous and bounded functions $g$. 
Now, to prove Theorem \ref{thm:2} it suffices to show that
\begin{align}
\esp_{\calP_n} \pp{f(W_n)g(Y_n)} - \esp_\calP \pp{f(SZ)g(SZ')} 
&= \esp_{\calP_n}  \pp{f(W_n)g(Y_n)} - \esp_{\calP_n} \pp{ f(SZ) g(Y_n)}\nonumber\\
 & \quad +   \esp_{\calP_n}\pp{ f(SZ)g(Y_n)}- \esp_\calP \pp{f(SZ) g(SZ')}\label{eq:joint conv cond}
\end{align}
tends to zero as $n\to\infty$ for all $f,g$ as above \citep[Corollary 1.4.5]{vandervaart96weak}. 
The absolute value of the first difference on the right-hand side above is the same as
\begin{align*}
\abs{\esp_{\calP_n}\pp{g(Y_n)(f(W_n)-f(SZ))}} & = 
\abs{\esp_{\calP_n}\pp{g(Y_n)\esp_\calP(f(W_n)-f(SZ))}}\\
& \le \nn g_\infty \esp_{\calP_n}\abs{\pp{\esp_\calP(f(W_n)-f(SZ))}}\to 0,
\end{align*}
almost surely, where the equality follows since $Y_n$ is $\calP$-measurable. The convergence follows from the conditional dominated convergence applied to
\[
X_n:=\abs{\esp_\calP(f(W_n)-f(SZ))},
\]
since $X_n\to0$ almost surely and $X_n\leq2\nn f_\infty$.

Recall $S\equiv S_\alpha^{1/2} = (\Gamma(1-\alpha)D)^{1/2}$. Introduce $S_n= (\Gamma(1-\alpha)D_n)^{1/2}$ (recall $D_n$ is $\calP_n$-measurable as introduced in \eqref{eq:D_n}). For the second term on the right-hand side of \eqref{eq:joint conv cond}, we decompose further as
\begin{align*}
 \esp_{\calP_n}\pp{ f(SZ)g(Y_n)}- \esp_\calP \pp{f(SZ) g(SZ')}
&= \esp_{\calP_n}  \pp{ f(SZ)g(Y_n)}- \esp_{\calP_n} \pp{f(S_nZ) g(Y_n)} \\
& \quad + \esp_{\calP_n}f(S_nZ)\esp_{\calP_n}g(Y_n)- \esp_{\calP_n}f(S_nZ) \esp_\calP g(SZ')\\
&\quad +\esp_{\calP}f(S_nZ)\esp_\calP g(SZ')- \esp_\calP f(SZ) \esp_\calP g(SZ').
\end{align*}
Here we used
\[
\esp_{\calP_n}f(S_nZ)=\esp_\calP f(S_nZ),
\]
which follows since $S_n$ is $\calP_n$-measurable and $Z$ is independent from $\calP$.
We have for the first difference,
\[
\abs{\esp_{\calP_n}  \pp{ f(SZ)g(Y_n)}- \esp_{\calP_n} \pp{f(S_nZ) g(Y_n)}}\le \nn g_\infty\esp_{\calP_n}\abs{f(SZ)-f(S_nZ)}\to 0
\]
almost surely, by the conditional dominated convergence applied to
\[
X_n:=\abs{f(SZ)-f(S_nZ)}.
\]
Indeed, $X_n\to0$ almost surely and $X_n\leq2\nn f_\infty$. The second difference goes to zero by Proposition \ref{prop:CLT Y}, and the third difference goes to zero since $S_n\to S$, by the dominated convergence theorem for the fixed conditioning $\sigma$-algebra $\calP$. This completes the proof.
\end{proof}

\begin{acks}[Acknowledgments]
Y.W.~thanks Alexander Iksanov for pointing out several closely related references in Remark \ref{rem:literature}. Y.W.~thanks the Associate Editor and an anonymous referee for helpful suggestions and Vishakha for discussions.
\end{acks}

\begin{funding}
Y.W.~was partially supported by Simons Foundation (MP-TSM-00002359). 
\end{funding}

\bibliographystyle{imsart-nameyear} 
	\bibliography{../../include/references,../../include/references18}

\end{document}